\documentclass[11pt]{amsart}
\usepackage[T1]{fontenc}
\usepackage{amsmath,amssymb,amsthm,booktabs}
\usepackage[hidelinks]{hyperref}
\usepackage[capitalize,nameinlink]{cleveref}

\theoremstyle{plain}
\newtheorem{theorem}{Theorem}[section]
\newtheorem{proposition}[theorem]{Proposition}
\newtheorem{lemma}[theorem]{Lemma}
\newtheorem{corollary}[theorem]{Corollary}
\newtheorem{conjecture}[theorem]{Conjecture}
\theoremstyle{remark}
\newtheorem{remark}[theorem]{Remark}

\newcommand{\FD}{\mathrm{FD}}
\newcommand{\FO}{\mathrm{FO}}

\newcommand{\E}{\mathbb{E}}

\title{Fixed-perimeter Franklin statistics and an eventual inequality}
\author{Marcus McCrea}
\address{Calgary, Alberta, Canada}

\begin{document}

\begin{abstract}
Gray, Payne, Swisher, and Watson conjectured that, for fixed integers $j\ge 0$ and $k\ge 2$, the number $\FD_{j,k}(n)$ of partitions of perimeter $n$ having exactly $j$ part sizes of multiplicity at least $k$ is eventually at least the number $\FO_{j,k}(n)$ having exactly $j$ distinct part sizes divisible by $k$. We derive compact bivariate generating functions for both statistics from the profile-word encoding of a partition. For $k\ge 3$, coefficient extraction shows that fixing $j$ changes only the order of the dominant pole, not its location. The two dominant poles are positive real numbers $\rho_k$ and $\sigma_k$ determined by
\[
\rho_k+\rho_k^2+\cdots+\rho_k^k=1,\qquad \sigma_k^k=(1-\sigma_k)^{k-1}.
\]
We prove $\rho_k<\sigma_k$ for every $k\ge 3$ and give positive, explicit constants $C_{j,k}$ such that
\[
\frac{\FD_{j,k}(n)}{\FO_{j,k}(n)}\sim C_{j,k}\Bigl(\frac{\sigma_k}{\rho_k}\Bigr)^{n}.
\]
This proves the conjecture and gives a strict eventual inequality for $k\ge 3$. We also determine the large-$k$ scale of the root gap, tabulate how large $n$ must be for the inequality to set in, and pose a single-crossover conjecture for the regime in which the number of marked part sizes is proportional to the perimeter, supported by exact computations and a large-deviation heuristic. The case $k=2$ recovers the known exact identity.
\end{abstract}

\maketitle

\section{Introduction}

For a partition $\pi$, let $\alpha(\pi)$ be its largest part and let $\lambda(\pi)$ be its number of parts. Its \emph{perimeter} is
\[
\Gamma(\pi)=\alpha(\pi)+\lambda(\pi)-1,
\]
the largest hook length of its Ferrers diagram. Fixed-perimeter partition identities were initiated by Straub and developed further by Fu and Tang, Lin--Xiong--Yan, Waldron, and others; see \cite{Straub,FuTang,LinXiongYan,Waldron}. Fixed-perimeter refinements of Beck-type identities appear in \cite{AAB}, building on the Beck-type companion identities of Ballantine and Welch \cite{BW1,BW2}, and fixed-perimeter parity biases are studied in \cite{GPSW2,Mahanta}.

For a positive integer $s$, write $m_s(\pi)$ for the multiplicity of the part size $s$ in $\pi$. For integers $j\ge0$ and $k\ge2$ define
\begin{align*}
\FO_{j,k}(n)&=\#\{\pi:\Gamma(\pi)=n,\ \#\{s\ge1: k\mid s,\ m_s(\pi)>0\}=j\},\\
\FD_{j,k}(n)&=\#\{\pi:\Gamma(\pi)=n,\ \#\{s\ge1: m_s(\pi)\ge k\}=j\}.
\end{align*}
Thus $\FO_{j,k}(n)$ counts partitions of perimeter $n$ with exactly $j$ distinct occurring part sizes divisible by $k$, while $\FD_{j,k}(n)$ counts those with exactly $j$ part sizes occurring at least $k$ times.

These are the fixed-perimeter analogues of the statistics in Franklin's classical identity \cite{Franklin}: for ordinary partitions of $n$, the number with exactly $j$ part sizes divisible by $k$ equals the number with exactly $j$ part sizes of multiplicity at least $k$, for all $j\ge0$, $k\ge2$, $n\ge1$. Gray, Payne, Swisher, and Watson proved that the perimeter analogue holds for $k=2$, namely $\FO_{j,2}(n)=\FD_{j,2}(n)$ for all $j,n\ge0$ \cite[Theorem~1.1]{GPSW}, observed that it fails for $k\ge3$, and conjectured the following \cite[Conjecture~1.2]{GPSW}.

\begin{conjecture}[Fixed-perimeter Franklin conjecture]\label{conj:GPSW}
For $j\ge0$ and $k\ge2$, $\FD_{j,k}(n)\ge\FO_{j,k}(n)$ for all sufficiently large $n$.
\end{conjecture}

An independent proof of this conjecture by Cuthbertson \cite{Cuthbertson} appeared after the first version of this paper. That proof derives quadrivariate generating functions and also obtains leading asymptotics. Here the profile weights give bivariate generating functions directly; a period of $k$ possible part sizes collapses into one denominator on the $\FO$-side. The resulting coefficient formulas isolate the two roots without retaining arm and leg variables.

We make the exponential comparison explicit. The polynomials $H_k$, $K_k$ and the series $U,T$ are defined in \cref{prop:D,prop:O}, and $\rho_k,\sigma_k$ in \eqref{eq:rho} and \eqref{eq:sigma}.

\begin{theorem}\label{thm:main}
For every fixed $j\ge0$ and $k\ge3$, there is an explicitly given constant $C_{j,k}>0$ such that
\[
\frac{\FD_{j,k}(n)}{\FO_{j,k}(n)}=C_{j,k}\Bigl(\frac{\sigma_k}{\rho_k}\Bigr)^{n}\bigl(1+O_{j,k}(n^{-1})\bigr).
\]
In particular, the ratio tends to infinity and $\FD_{j,k}(n)>\FO_{j,k}(n)$ for all sufficiently large $n$.
\end{theorem}

For fixed $j$, marking $j$ part sizes raises the pole order by $j$ without changing its location. The root gap then decides the eventual inequality. The constants $C_{j,k}$ are less than $1$ in every case we computed, so numerically the inequality sets in only for $n$ of order $j/c_k$ for a constant $c_k$ depending on $k$ (\cref{sec:threshold,sec:further}); in particular no threshold independent of $j$ exists. The two-variable generating functions and the root-gap estimate keep this mechanism visible throughout the proof.

\section{Profile words and two bivariate generating functions}

We use the profile encoding employed in \cite{FuTang,GPSW}. A partition with largest part $\alpha$ can be written uniquely as
\begin{equation}\label{eq:pi}
\pi=1^{m_1}2^{m_2}\cdots(\alpha-1)^{m_{\alpha-1}}\alpha^{m_\alpha+1},\qquad m_1,\dots,m_\alpha\ge0.
\end{equation}
Its profile word is
\begin{equation}\label{eq:profile}
EN^{m_1}EN^{m_2}\cdots EN^{m_\alpha}N.
\end{equation}
Since $\lambda(\pi)=1+\sum_{i=1}^\alpha m_i$, the perimeter is
\begin{equation}\label{eq:perim}
\Gamma(\pi)=\alpha+\sum_{i=1}^{\alpha}m_i.
\end{equation}
Accordingly, in the generating functions below, a block $EN^m$ receives weight $q^{m+1}$ and the final $N$ receives weight $1$.

Define
\[
D_k(z,q)=\sum_{n\ge1}\sum_{j\ge0}\FD_{j,k}(n)z^jq^n,\qquad
O_k(z,q)=\sum_{n\ge1}\sum_{j\ge0}\FO_{j,k}(n)z^jq^n.
\]

\begin{proposition}\label{prop:Dgf}
For every $k\ge2$,
\begin{equation}\label{eq:D}
D_k(z,q)=\frac{q\bigl(1-(1-z)q^{k-1}\bigr)}{1-2q+(1-z)q^{k+1}}.
\end{equation}
\end{proposition}

\begin{proof}
For $i<\alpha$, the multiplicity of the part size $i$ in \eqref{eq:pi} is $m_i$. Hence the block $EN^{m_i}$ is marked by $z$ exactly when $m_i\ge k$. Its generating function is
\[
I_k(z,q)=q\sum_{m\ge0}q^mz^{[m\ge k]}
=q\Bigl(1+q+\cdots+q^{k-1}+z\frac{q^k}{1-q}\Bigr)
=\frac{q\bigl(1-(1-z)q^k\bigr)}{1-q}.
\]
The largest part $\alpha$ has multiplicity $m_\alpha+1$, so the terminal block $EN^{m_\alpha}$ is marked exactly when $m_\alpha\ge k-1$. Its generating function is
\[
L_k(z,q)=q\sum_{m\ge0}q^mz^{[m\ge k-1]}=\frac{q\bigl(1-(1-z)q^{k-1}\bigr)}{1-q}.
\]
A profile consists of an arbitrary sequence of interior blocks followed by one terminal block. Therefore
\[
D_k(z,q)=L_k(z,q)\sum_{r\ge0}I_k(z,q)^r=\frac{L_k(z,q)}{1-I_k(z,q)}.
\]
Substitution and simplification give \eqref{eq:D}.
\end{proof}

The next formula is the key compression: rather than treating residue classes one block at a time, we sum a complete period of $k$ possible part sizes.

\begin{proposition}\label{prop:Ogf}
Let $A=A(q)=q/(1-q)$. For every $k\ge2$,
\begin{equation}\label{eq:O}
O_k(z,q)=\frac{A+A^2+\cdots+A^{k-1}+zA^k}{1-A^k\bigl(1-(1-z)q\bigr)}.
\end{equation}
\end{proposition}

\begin{proof}
An unmarked block $EN^m$ has generating function $A=q\sum_{m\ge0}q^m=q/(1-q)$. Suppose $i<\alpha$ and $k\mid i$. The part size $i$ occurs exactly when $m_i\ge1$. Thus the corresponding block has generating function
\[
q\bigl(1+zq+zq^2+\cdots\bigr)=\frac{q\bigl(1-(1-z)q\bigr)}{1-q}=A\bigl(1-(1-z)q\bigr).
\]
Set $s=1-(1-z)q$. Hence each interior position divisible by $k$ contributes an extra factor $s$ relative to an ordinary block.

We now condition on the largest part $\alpha$. If $\alpha=tk+r$ with $t\ge0$ and $1\le r\le k-1$, then exactly $t$ integers among $1,\dots,\alpha-1$ are divisible by $k$, and the largest part is not divisible by $k$. The profile weight is therefore $A^{tk+r}s^t$. If instead $\alpha=tk$ with $t\ge1$, then exactly $t-1$ interior positions are divisible by $k$. The largest part itself is divisible by $k$ and necessarily occurs, so it contributes the marking factor $z$. The profile weight is $zA^{tk}s^{t-1}$. These two cases partition the positive integers $\alpha$, and hence
\[
O_k(z,q)=\sum_{t\ge0}\sum_{r=1}^{k-1}A^{tk+r}s^t+\sum_{t\ge1}zA^{tk}s^{t-1}
=\frac{A+A^2+\cdots+A^{k-1}}{1-A^ks}+\frac{zA^k}{1-A^ks},
\]
which is \eqref{eq:O}.
\end{proof}

\begin{remark}\label{rem:k2}
Setting $z=1$ in either \eqref{eq:D} or \eqref{eq:O} gives
\[
\frac{q}{1-2q}=\sum_{n\ge1}2^{n-1}q^n,
\]
the generating function for all nonempty partitions by perimeter. When $k=2$, direct simplification of \eqref{eq:D} and \eqref{eq:O} gives the same bivariate rational function
\[
\frac{q\bigl(1-(1-z)q\bigr)}{1-2q+(1-z)q^3},
\]
recovering the exact identity $\FO_{j,2}(n)=\FD_{j,2}(n)$ of \cite{GPSW}.
\end{remark}

\section{Fixing the number of marked part sizes}

For fixed $j$ and $k$, write
\[
D_{j,k}(q)=\sum_{n\ge1}\FD_{j,k}(n)q^n,\qquad O_{j,k}(q)=\sum_{n\ge1}\FO_{j,k}(n)q^n.
\]
The following formulas make the dominant-pole structure explicit.

\begin{proposition}\label{prop:D}
Let $H_k(q)=1-2q+q^{k+1}$. Then
\begin{equation}\label{eq:D0}
D_{0,k}(q)=\frac{q(1-q^{k-1})}{H_k(q)},
\end{equation}
and, for $j\ge1$,
\begin{equation}\label{eq:Dj}
D_{j,k}(q)=\frac{q^{(k+1)j-1}(1-q)^2}{H_k(q)^{j+1}}.
\end{equation}
\end{proposition}

\begin{proof}
Equation \eqref{eq:D} can be rewritten as
\[
D_k(z,q)=\frac{q\bigl(1-q^{k-1}+zq^{k-1}\bigr)}{H_k(q)-zq^{k+1}}.
\]
The coefficient of $z^0$ is \eqref{eq:D0}. For $j\ge1$, expand
\[
\frac1{H_k-zq^{k+1}}=\sum_{r\ge0}\frac{z^rq^{(k+1)r}}{H_k^{r+1}}.
\]
Thus
\[
D_{j,k}(q)=\frac{q(1-q^{k-1})q^{(k+1)j}}{H_k^{j+1}}+\frac{q^k\,q^{(k+1)(j-1)}}{H_k^{j}}
=\frac{q^{(k+1)j-1}}{H_k^{j+1}}\Bigl(q^2(1-q^{k-1})+H_k\Bigr).
\]
Since $q^2(1-q^{k-1})+H_k(q)=q^2-q^{k+1}+1-2q+q^{k+1}=(1-q)^2$, we obtain \eqref{eq:Dj}.
\end{proof}

For the $\FO$-side, define
\[
U=A+A^2+\cdots+A^{k-1},\quad V=A^k,\quad K_k(q)=1-A^k(1-q),\quad T(q)=A^kq.
\]

\begin{proposition}\label{prop:O}
With the notation above,
\begin{equation}\label{eq:O0}
O_{0,k}(q)=\frac{U}{K_k},
\end{equation}
and, for $j\ge1$,
\begin{equation}\label{eq:Oj}
O_{j,k}(q)=\frac{UT^j}{K_k^{j+1}}+\frac{VT^{j-1}}{K_k^{j}}.
\end{equation}
\end{proposition}

\begin{proof}
Since $1-A^k\bigl(1-(1-z)q\bigr)=K_k-zT$, we have $O_k(z,q)=(U+zV)/(K_k-zT)$. Expanding $(K_k-zT)^{-1}$ geometrically in $z$ gives \eqref{eq:O0} and \eqref{eq:Oj} by coefficient extraction.
\end{proof}

The common feature of \cref{prop:D,prop:O} is that fixing $j$ raises the order of a single basic denominator without moving its smallest positive zero.

\section{The two dominant roots}

The $\FD$ denominator factors as
\begin{equation}\label{eq:Hfactor}
H_k(q)=(1-q)\bigl(1-q-q^2-\cdots-q^k\bigr).
\end{equation}
Let $\rho_k\in(0,1)$ be the unique positive solution of
\begin{equation}\label{eq:rho}
\rho_k+\rho_k^2+\cdots+\rho_k^k=1.
\end{equation}
On the $\FO$-side,
\[
K_k(q)=1-\frac{q^k}{(1-q)^{k-1}}=\frac{\widetilde K_k(q)}{(1-q)^{k-1}},\qquad \widetilde K_k(q):=(1-q)^{k-1}-q^k.
\]
Let $\sigma_k\in(0,1)$ be the unique positive solution of
\begin{equation}\label{eq:sigma}
\sigma_k^k=(1-\sigma_k)^{k-1}.
\end{equation}
Existence and uniqueness of both positive roots follow from strict monotonicity of $S_k(x)=x+\cdots+x^k$ and of $g_k(x)=x^k/(1-x)^{k-1}$ on $(0,1)$. Note that $\widetilde K_k(1)=-1\neq0$, so the zeros of the polynomial $\widetilde K_k$ are exactly the zeros of $K_k$.

\begin{lemma}[Root gap]\label{lem:gap}
For every $k\ge3$, $\rho_k<\sigma_k$.
\end{lemma}

\begin{proof}
Fix $k\ge3$ and write $\rho=\rho_k$. Since
\[
S_k\bigl(\tfrac12\bigr)=1-2^{-k}<1\quad\text{and}\quad
S_k\bigl(\tfrac35\bigr)>\tfrac35+\tfrac9{25}+\tfrac{27}{125}=\tfrac{147}{125}>1,
\]
strict monotonicity gives
\begin{equation}\label{eq:rhobounds}
\tfrac12<\rho<\tfrac35.
\end{equation}
From \eqref{eq:rho}, the geometric-sum identity gives $\rho(1-\rho^k)/(1-\rho)=1$, so
\begin{equation}\label{eq:rhoid}
1-\rho=\rho(1-\rho^k).
\end{equation}
Then $g_k$ is strictly increasing with $g_k(\sigma_k)=1$. By \eqref{eq:rhoid},
\begin{equation}\label{eq:grho}
g_k(\rho)=\frac{\rho}{(1-\rho^k)^{k-1}}.
\end{equation}
Bernoulli's inequality yields
\begin{equation}\label{eq:bern}
(1-\rho^k)^{k-1}\ge1-(k-1)\rho^k.
\end{equation}
We show that the right-hand side of \eqref{eq:bern} is larger than $\rho$. By \eqref{eq:rhoid}, the inequality $1-(k-1)\rho^k>\rho$ is equivalent to
\begin{equation}\label{eq:17}
(k-1)\rho^{k-1}+\rho^k<1.
\end{equation}
Using \eqref{eq:rhobounds},
\[
(k-1)\rho^{k-1}+\rho^k<\Bigl(\tfrac35\Bigr)^{k-1}\Bigl(k-1+\tfrac35\Bigr)=:b_k.
\]
Then $b_3=117/125<1$, and for $k\ge3$,
\[
\frac{b_{k+1}}{b_k}=\frac35\cdot\frac{5k+3}{5k-2}<1,
\]
because $15k+9<25k-10$. Hence $b_k<1$ for every $k\ge3$, proving \eqref{eq:17}. Combining this with \eqref{eq:bern} gives $(1-\rho^k)^{k-1}>\rho$. Equation \eqref{eq:grho} now implies $g_k(\rho)<1$. Since $g_k$ is strictly increasing and $g_k(\sigma_k)=1$, we conclude $\rho<\sigma_k$.
\end{proof}

(For $k=2$ the two roots coincide, both equal to $(\sqrt5-1)/2$; the bound \eqref{eq:rhobounds} fails there, which is why the lemma starts at $k=3$.)

We next verify that no complex zero of equal modulus competes with either positive root.

\begin{lemma}\label{lem:rhosimple}
The number $\rho_k$ is the unique zero of $H_k(q)$ of modulus at most $\rho_k$, and it is a simple zero.
\end{lemma}

\begin{proof}
Because $\rho_k<1$, the factor $1-q$ in \eqref{eq:Hfactor} does not vanish on $|q|\le\rho_k$. Suppose $1-z-z^2-\cdots-z^k=0$ with $|z|\le\rho_k$. Then
\[
1=|z+z^2+\cdots+z^k|\le|z|+|z|^2+\cdots+|z|^k\le\rho_k+\rho_k^2+\cdots+\rho_k^k=1.
\]
Equality holds throughout. Thus $|z|=\rho_k$, and equality in the triangle inequality forces $z,z^2,\dots,z^k$ to have the same argument. In particular $z$ and $z^2$ have the same argument. Since $z\ne0$, this forces $z$ to be positive real, so $z=\rho_k$. Finally,
\[
\frac{d}{dq}\bigl(1-q-q^2-\cdots-q^k\bigr)=-\sum_{i=1}^k iq^{i-1}<0
\]
at $q=\rho_k$, so the zero is simple.
\end{proof}

\begin{lemma}\label{lem:sigmasimple}
The number $\sigma_k$ is the unique zero of $\widetilde K_k$ (equivalently of $K_k$) of modulus at most $\sigma_k$, and it is a simple zero.
\end{lemma}

\begin{proof}
If $\widetilde K_k(z)=0$, then $z^k=(1-z)^{k-1}$, and hence $|z|^k=|1-z|^{k-1}$. No zero can satisfy $|z|<\sigma_k$: indeed, if $|z|<\sigma_k<1$, then $|1-z|\ge1-|z|$ gives
\[
1=\frac{|z|^k}{|1-z|^{k-1}}\le\frac{|z|^k}{(1-|z|)^{k-1}}<\frac{\sigma_k^k}{(1-\sigma_k)^{k-1}}=1,
\]
a contradiction. If $|z|=\sigma_k$, the same inequalities must be equalities. Equality in $|1-z|\ge1-|z|$ forces $z$ to lie on the nonnegative real axis. Therefore $z=\sigma_k$. The function $q^k/(1-q)^{k-1}$ has strictly positive derivative on $(0,1)$, so $K_k'(\sigma_k)<0$, and $\widetilde K_k'(\sigma_k)=(1-\sigma_k)^{k-1}K_k'(\sigma_k)<0$. Thus the zero is simple.
\end{proof}

\section{Coefficient growth and proof of the theorem}

We record the following coefficient estimate, including its error term. It is the standard consequence of a unique dominant pole and partial fractions (compare \cite[Thm.~IV.9]{FlajoletSedgewick}); we include the proof to make the hypotheses explicit. We allow the numerator to be rational, since the $\FO$-side numerators have a pole at $q=1$.

\begin{lemma}\label{lem:coef}
Let $Q$ be a polynomial, let $m\ge1$ be an integer, and let $r>0$ satisfy $Q(r)=0$ and $Q'(r)<0$, with no other zero of $Q$ in $|q|\le r$. Let $P$ be a rational function with no pole in $|q|\le r$ and $P(r)>0$. For $F(q)=P(q)/Q(q)^m$,
\begin{equation}\label{eq:coef}
[q^n]F(q)=\frac{P(r)}{(-rQ'(r))^m}\,\frac{n^{m-1}}{(m-1)!}\,r^{-n}\bigl(1+O(n^{-1})\bigr).
\end{equation}
\end{lemma}

\begin{proof}
Because $r$ is a simple zero of $Q$, we can write $Q(q)=B(q)(1-q/r)$ with $B$ a polynomial whose zeros are the zeros of $Q$ other than $r$; differentiating at $q=r$ gives $B(r)=-rQ'(r)>0$. Hence
\[
F(q)=C(q)\Bigl(1-\frac qr\Bigr)^{-m},\qquad C(q)=\frac{P(q)}{B(q)^m},\qquad C(r)=\frac{P(r)}{(-rQ'(r))^m}.
\]
$F$ is rational, and its only pole in $|q|\le r$ is at $q=r$, of order $m$. Its partial-fraction decomposition consists of a polynomial and principal parts at the finitely many poles. The principal part at $r$ has the form $\sum_{s=1}^m c_s(1-q/r)^{-s}$ with $c_m=C(r)$. The coefficient of $q^n$ in $c_m(1-q/r)^{-m}$ is
\[
C(r)\binom{n+m-1}{m-1}r^{-n}=C(r)\frac{n^{m-1}}{(m-1)!}\,r^{-n}\bigl(1+O(n^{-1})\bigr),
\]
and the terms with $s<m$ contribute $O(n^{m-2}r^{-n})$ (and are absent for $m=1$). Every other pole has modulus larger than $r$; choose $R>r$ smaller than all of them. Their principal parts, including any polynomial factors from higher-order poles, contribute $O(n^{M}R^{-n})$ for some $M$, which is $O(n^{-1}r^{-n})$ relative to the main term. The polynomial part affects only finitely many coefficients. This proves \eqref{eq:coef}.
\end{proof}

\begin{proposition}[Leading constants]\label{prop:const}
For $k\ge2$, put
\[
d_k=-\rho_kH_k'(\rho_k)>0,\qquad e_k=-\sigma_kK_k'(\sigma_k)>0.
\]
Then, for each fixed $j\ge0$,
\begin{align}
\FD_{j,k}(n)&=D^*_{j,k}\,n^j\rho_k^{-n}\bigl(1+O_{j,k}(n^{-1})\bigr),\label{eq:FDasym}\\
\FO_{j,k}(n)&=O^*_{j,k}\,n^j\sigma_k^{-n}\bigl(1+O_{j,k}(n^{-1})\bigr),\label{eq:FOasym}
\end{align}
where
\begin{equation}\label{eq:Dstar}
D^*_{0,k}=\frac{\rho_k(1-\rho_k^{k-1})}{d_k},\qquad
D^*_{j,k}=\frac{\rho_k^{(k+1)j-1}(1-\rho_k)^2}{j!\,d_k^{\,j+1}}\quad(j\ge1),
\end{equation}
\begin{equation}\label{eq:Ostar}
O^*_{0,k}=\frac{U(\sigma_k)}{e_k},\qquad
O^*_{j,k}=\frac{U(\sigma_k)T(\sigma_k)^j}{j!\,e_k^{\,j+1}}\quad(j\ge1).
\end{equation}
\end{proposition}

\begin{proof}
\emph{FD-side.} By \cref{prop:D,lem:rhosimple}, the sole dominant pole of $D_{0,k}$ is simple at $\rho_k$, with numerator $\rho_k(1-\rho_k^{k-1})>0$. For $j\ge1$, \eqref{eq:Dj} has a pole of exact order $j+1$ there, with numerator $\rho_k^{(k+1)j-1}(1-\rho_k)^2>0$. Since $H_k(q)=(1-q)(1-\sum_{i=1}^kq^i)$, we have $H_k'(\rho_k)=-(1-\rho_k)\sum_{i=1}^ki\rho_k^{i-1}<0$. \Cref{lem:coef} with $Q=H_k$, $m=j+1$ gives \eqref{eq:FDasym} and \eqref{eq:Dstar}.

\emph{FO-side.} The functions $A,U,V,T$ are rational with their only pole at $q=1$, and $\sigma_k<1$. Put
\[
P_0=U,\qquad P_j=UT^j+VT^{j-1}K_k\quad(j\ge1),
\]
so that $O_{j,k}=P_j/K_k^{j+1}$ by \cref{prop:O}. Writing $K_k=\widetilde K_k/(1-q)^{k-1}$ we get $O_{j,k}=\widetilde P_j/\widetilde K_k^{\,j+1}$ with $\widetilde P_j=(1-q)^{(k-1)(j+1)}P_j$, a rational function without poles in $|q|\le\sigma_k$. By \cref{lem:sigmasimple}, $\sigma_k$ is the only zero of $\widetilde K_k$ in $|q|\le\sigma_k$ and it is simple. Since $K_k(\sigma_k)=0$, we have $\widetilde K_k'(\sigma_k)=(1-\sigma_k)^{k-1}K_k'(\sigma_k)$ and $P_j(\sigma_k)=U(\sigma_k)T(\sigma_k)^j>0$. \Cref{lem:coef} with $Q=\widetilde K_k$, $m=j+1$ gives the constant
\[
\frac{\widetilde P_j(\sigma_k)}{\bigl(-\sigma_k\widetilde K_k'(\sigma_k)\bigr)^{j+1}}
=\frac{(1-\sigma_k)^{(k-1)(j+1)}P_j(\sigma_k)}{\bigl((1-\sigma_k)^{k-1}e_k\bigr)^{j+1}}
=\frac{P_j(\sigma_k)}{e_k^{\,j+1}},
\]
which is \eqref{eq:Ostar}, and \eqref{eq:FOasym} follows.
\end{proof}

\begin{proof}[Proof of \cref{thm:main}]
Fix $j\ge0$ and $k\ge3$. By \cref{prop:const}, both leading constants are positive. Dividing \eqref{eq:FDasym} by \eqref{eq:FOasym} gives
\[
\frac{\FD_{j,k}(n)}{\FO_{j,k}(n)}=\frac{D^*_{j,k}}{O^*_{j,k}}\Bigl(\frac{\sigma_k}{\rho_k}\Bigr)^{n}\bigl(1+O_{j,k}(n^{-1})\bigr).
\]
Take $C_{j,k}=D^*_{j,k}/O^*_{j,k}$. By \cref{lem:gap}, $\sigma_k/\rho_k>1$, so the ratio tends to infinity.
\end{proof}

\begin{corollary}\label{cor:ev}
For every $j\ge0$ and $k\ge2$ there is an integer $N=N(j,k)$ such that $\FD_{j,k}(n)\ge\FO_{j,k}(n)$ for $n\ge N$. For $k\ge3$ the inequality is strict for all sufficiently large $n$.
\end{corollary}

\begin{proof}
For $k=2$, equality follows from \cref{rem:k2}, equivalently from \cite[Theorem~1.1]{GPSW}. For $k\ge3$, \cref{thm:main} implies that the ratio $\FD_{j,k}(n)/\FO_{j,k}(n)$ is eventually greater than $1$.
\end{proof}

\section{Closed forms, consistency checks, and effective thresholds}\label{sec:threshold}

\subsection{Closed forms on the FO-side}
At $q=\sigma_k$ the defining equation \eqref{eq:sigma} reads $A^k(1-\sigma_k)=1$, where $A=\sigma_k/(1-\sigma_k)$. This simplifies the constants.

\begin{proposition}\label{prop:closed}
Write $\sigma=\sigma_k$. Then
\[
V(\sigma)=\frac1{1-\sigma},\qquad T(\sigma)=\frac{\sigma}{1-\sigma},\qquad
U(\sigma)=\frac{1-\sigma}{2\sigma-1},\qquad e_k=k+(k-1)\frac{\sigma}{1-\sigma}.
\]
Consequently $O^*_{j,k}=\dfrac{1-\sigma}{2\sigma-1}\cdot\dfrac{1}{j!\,e_k^{\,j+1}}\Bigl(\dfrac{\sigma}{1-\sigma}\Bigr)^{j}$ for all $j\ge0$.
\end{proposition}

\begin{proof}
The first two identities follow from $A^k=1/(1-\sigma)$. Next, $U=(A-A^k)/(1-A)$ with $A-A^k=\bigl(\sigma-1\bigr)/(1-\sigma)=-1$ and $1-A=(1-2\sigma)/(1-\sigma)$, giving $U=(1-\sigma)/(2\sigma-1)$, which is positive since $\sigma_k>\tfrac12$. Finally $K_k=1-g_k$ with $g_k(\sigma)=1$ and $g_k'/g_k=k/q+(k-1)/(1-q)$, so $e_k=\sigma g_k'(\sigma)=k+(k-1)\sigma/(1-\sigma)$.
\end{proof}

For example, for $k=3$ one finds $\sigma_3=0.569840\ldots$, $e_3=5.6494\ldots$, $O^*_{0,3}=0.5451\ldots$ and $O^*_{1,3}=0.1278\ldots$, in agreement with exact coefficients.

\begin{remark}[Consistency with $k=2$]
For $k=2$ one has $\rho_2=\sigma_2=(\sqrt5-1)/2$, and the constants $D^*_{j,2}$ and $O^*_{j,2}$ coincide (numerically $0.4472,\ 0.2,\ 0.04472$ for $j=0,1,2$), as the exact identity requires.
\end{remark}

\subsection{How large must the perimeter be?}
\Cref{thm:main} is an eventual statement. Since $C_{j,k}<1$ in all cases we computed (and appears to decay rapidly in $j$), the inequality $\FD_{j,k}(n)>\FO_{j,k}(n)$ sets in only once $n$ exceeds roughly
\[
n^*_{j,k}=\frac{\log(1/C_{j,k})}{\log(\sigma_k/\rho_k)}.
\]
\Cref{tab:thr} compares $n^*_{j,k}$ with the exact last $n\le700$ at which $\FD_{j,k}(n)<\FO_{j,k}(n)$, computed from the exact rational generating functions. For $j=0$ the inequality $\FD_{0,k}(n)\ge\FO_{0,k}(n)$ held for every $n$ tested.

\begin{table}[h]
\centering
\begin{tabular}{lcccc}
\toprule
 & $j=1$ & $j=2$ & $j=3$ & $j=4$\\
\midrule
$k=3$: $n^*$ / last failure & 15.6 / 14 & 30.1 / 28 & 44.7 / 41 & 59.2 / 55\\
$k=5$: $n^*$ / last failure & 35.4 / 34 & 67.8 / 66 & 100.1 / 98 & 132.5 / 130\\
$k=8$: $n^*$ / last failure & 88.1 / 87 & 170.1 / 169 & 252.1 / 250 & 334.1 / 332\\
\bottomrule
\end{tabular}
\caption{Heuristic threshold $n^*_{j,k}$ versus the exact last failure of $\FD_{j,k}(n)\ge\FO_{j,k}(n)$.}\label{tab:thr}
\end{table}

Two observations follow. First, the threshold grows roughly linearly in $j$, with a slope that increases with $k$. Second, no threshold independent of $j$ can exist: for $k=3$ and $n=200$ an exact computation gives $\FD_{j,3}(200)>\FO_{j,3}(200)$ for $j\le14$ but $\FD_{j,3}(200)<\FO_{j,3}(200)$ for $15\le j\le 50$ (the difference is $0$ for $j>50$), so $N(15,3)>200$. This is the phenomenon addressed in \cref{sec:further}. The increments of $n^*_{j,k}$ per unit $j$ in \cref{tab:thr} (about $14.6$ and $32.4$ for $k=3,5$) are numerically close to $1/c_k$ ($\approx13.7$ and $31.9$) for the crossover density $c_k$ of \cref{conj:cross}; we do not attempt to make this precise.

\section{Dependence of the root gap on the modulus}

The exponential separation holds for every fixed $k\ge3$. Its rate approaches $1$ as $k$ grows.

\begin{proposition}\label{prop:asym}
As $k\to\infty$,
\begin{align}
\rho_k&=\tfrac12+2^{-k-2}+O(k\,4^{-k}),\label{eq:rhoasym}\\
\sigma_k&=\tfrac12+\frac{\log2}{4k}+O(k^{-2}),\label{eq:sigmaasym}\\
\log\frac{\sigma_k}{\rho_k}&=\frac{\log2}{2k}+O(k^{-2}).\label{eq:gapasym}
\end{align}
\end{proposition}

\begin{proof}
Put $\delta_k=\rho_k-1/2$. Since $H_k(\rho_k)=0$,
\[
2\delta_k=\rho_k^{k+1},\qquad \delta_k=2^{-k-2}(1+2\delta_k)^{k+1}.
\]
The bound $\rho_k<3/5$ from \eqref{eq:rhobounds} implies $\delta_k=O((3/5)^k)$, hence $k\delta_k\to0$. In the second equation this first yields $\delta_k=O(2^{-k})$, and then $(1+2\delta_k)^{k+1}=1+O(k2^{-k})$, proving \eqref{eq:rhoasym}.

Write $\sigma_k=1/2+\varepsilon_k$. The function $x^k/(1-x)^{k-1}$ equals $1/2$ at $x=1/2$, so $\varepsilon_k>0$. Taking logarithms of \eqref{eq:sigma} gives
\begin{equation}\label{eq:26}
\log2=k\log(1+2\varepsilon_k)-(k-1)\log(1-2\varepsilon_k)
=(k-1)\log\frac{1+2\varepsilon_k}{1-2\varepsilon_k}+\log(1+2\varepsilon_k).
\end{equation}
For $0<\varepsilon<1/2$, $\log\frac{1+2\varepsilon}{1-2\varepsilon}\ge4\varepsilon$ and $\log(1+2\varepsilon)>0$; thus \eqref{eq:26} implies $\varepsilon_k\le\log2/(4(k-1))=O(k^{-1})$. Taylor expansion of the first form in \eqref{eq:26} now gives
\[
\log2=(4k-2)\varepsilon_k-2\varepsilon_k^2+O(k\varepsilon_k^3)=4k\varepsilon_k+O(k^{-1}),
\]
proving \eqref{eq:sigmaasym}. Finally, apply $\log(1+x)=x+O(x^2)$ to $\sigma_k/\rho_k=(1+2\varepsilon_k)/(1+2\delta_k)$ to obtain \eqref{eq:gapasym}.
\end{proof}

\begin{remark}
Iterating the second equation in the proof gives the sharper expansion $\rho_k=\tfrac12+2^{-k-2}+(k+1)2^{-2k-3}+O(k^28^{-k})$; in particular $(\rho_k-\tfrac12-2^{-k-2})/(k4^{-k})\to\tfrac18$. By \eqref{eq:gapasym}, the ratio $\FD_{j,k}/\FO_{j,k}$ grows like $2^{n/(2k)}$, so the thresholds in \cref{tab:thr} scale like $k$ times $\log(1/C_{j,k})$.
\end{remark}

\section{Concluding remark}

The proof isolates a simple structural principle. For each fixed number $j$ of marked part sizes, coefficient extraction in the marking variable raises the order of the dominant pole by $j$ but leaves its location unchanged. Thus the eventual comparison is controlled entirely by the two unmarked profile kernels
\[
1-q-q^2-\cdots-q^k\qquad\text{and}\qquad 1-\frac{q^k}{(1-q)^{k-1}}.
\]
The root gap $\rho_k<\sigma_k$ then creates an exponential separation between the two fixed-$j$ statistics. It would be interesting to know whether other fixed-perimeter statistics encoded by periodic profile weights admit similar eventual comparisons through their zero-defect growth kernels.

\section{Further work: a linear-density crossover conjecture}\label{sec:further}

The argument of this paper keeps $j$ fixed while $n\to\infty$. Extracting $[z^j]$ then changes the order of the dominant pole but not its location. The situation is different when $j$ grows proportionally to $n$. The following conjecture gives a precise form to one possible reversal in that regime; it is posed as an open problem, not as a consequence of \cref{thm:main}.

For $k\ge3$ fixed, put
\[
a_k=2^{-k-1},\qquad b_k=\frac1{4k}.
\]
Then $a_k<b_k$, since $4k<2^{k+1}$ for $k\ge3$ (with equality $a_2=b_2=1/8$ at $k=2$, consistent with $\FD_{j,2}=\FO_{j,2}$).

\begin{proposition}\label{prop:means}
Let $X^D_{n,k}$ and $X^O_{n,k}$ denote the number of marked part sizes in the definitions of $\FD$ and $\FO$, for a uniformly random partition of perimeter $n$. Then
\[
\lim_{n\to\infty}\frac{\E X^D_{n,k}}{n}=a_k,\qquad \lim_{n\to\infty}\frac{\E X^O_{n,k}}{n}=b_k.
\]
\end{proposition}

\begin{proof}
There are $2^{n-1}$ partitions of perimeter $n$, and $\sum_{\pi}X(\pi)\,q^{\Gamma(\pi)}=\partial_zD_k|_{z=1}$ (resp.\ $\partial_zO_k|_{z=1}$). From \eqref{eq:D},
\[
\partial_zD_k\big|_{z=1}=\frac{q^k}{1-2q}+\frac{q^{k+2}}{(1-2q)^2}.
\]
The dominant term is the double pole at $q=1/2$, whose coefficient is $2^{-k-2}(n+1)2^n\bigl(1+O(n^{-1})\bigr)=2^{-k-1}n\,2^{n-1}(1+O(n^{-1}))$. Dividing by $2^{n-1}$ gives $a_k$.

From $O_k=(U+zV)/(K_k-zT)$, and since $K_k-T=1-A^k$ at $z=1$,
\[
\partial_zO_k\big|_{z=1}=\frac{V}{1-A^k}+\frac{(U+V)T}{(1-A^k)^2}.
\]
At $q=1/2$ we have $A=1$, $A'=4$, so $1-A^k=2k(1-2q)\bigl(1+O(1-2q)\bigr)$, while $(U+V)T\to k\cdot\tfrac12$. The double-pole coefficient is therefore $\frac{k/2}{(2k)^2}=\frac1{8k}$, giving $[q^n]\sim\frac{n2^n}{8k}=\frac n{4k}2^{n-1}$. Dividing by $2^{n-1}$ gives $b_k$.
\end{proof}

\begin{conjecture}[Linear-density single-crossover conjecture]\label{conj:cross}
For every fixed integer $k\ge3$, there is a unique transition density $c_k\in(a_k,b_k)$ with the following property. For each fixed $x\in(a_k,b_k)\setminus\{c_k\}$, there exists $N=N(k,x)$ such that, for all $n\ge N$,
\[
\begin{cases}
\FD_{\lfloor xn\rfloor,k}(n)>\FO_{\lfloor xn\rfloor,k}(n), & a_k<x<c_k,\\[2pt]
\FD_{\lfloor xn\rfloor,k}(n)<\FO_{\lfloor xn\rfloor,k}(n), & c_k<x<b_k.
\end{cases}
\]
\end{conjecture}

The conjecture concerns a single transition between the two limiting means; it does not assert a sign rule at $x=c_k$ or outside $(a_k,b_k)$. The fixed-$j$ theorem provides no uniform estimate for $j=\lfloor xn\rfloor$, so a proof needs a different coefficient-asymptotic argument.

\subsection*{Evidence}
\emph{Exact data.} At $n=200$, the differences $\FD_{j,k}(200)-\FO_{j,k}(200)$ are positive at $j=14,9,6,4$ and negative at $j=15,10,7,5$, for $k=3,4,5,6$ respectively. For each of these $k$ the sign pattern over all $j$ is a single switch from $+$ to $-$, located inside $(a_kn,b_kn)$.

\emph{Large-deviation heuristic.} For $z>0$ let $\rho_D(z)$ and $\rho_O(z)$ be the smallest positive roots of $1-2q+(1-z)q^{k+1}$ and $(1-q)^k-q^k\bigl(1-(1-z)q\bigr)$, respectively (these are the dominant singularities of $D_k(z,\cdot)$ and $O_k(z,\cdot)$). The exponential growth rate of $\FD_{\lfloor xn\rfloor,k}(n)$ is $\varphi_D(x)=\min_{t\in\mathbb R}\bigl[-\log\rho_D(e^t)-xt\bigr]$, and similarly $\varphi_O$ for $\FO$. Numerically, $\varphi_D-\varphi_O$ is positive at $x=a_k$ (where $\FD$ is at its mean and $\FO$ is in a large-deviation regime), negative at $x=b_k$, and changes sign exactly once on a fine grid of $(a_k,b_k)$, for $k=3,4,5,6,8$. The resulting predicted crossover densities are compared with the finite-$n$ crossings in \cref{tab:cross}.

\begin{table}[h]
\centering
\begin{tabular}{ccccc}
\toprule
$k$ & $a_k$ & $b_k$ & $c_k$ (saddle point) & crossing at $n=200$\\
\midrule
3 & 0.0625 & 0.0833 & 0.0730 & 0.0725\\
4 & 0.0313 & 0.0625 & 0.0465 & 0.0475\\
5 & 0.0156 & 0.0500 & 0.0314 & 0.0325\\
6 & 0.0078 & 0.0417 & 0.0222 & 0.0225\\
8 & 0.0020 & 0.0313 & 0.0122 & ---\\
\bottomrule
\end{tabular}
\caption{Saddle-point prediction of $c_k$ versus the exact sign change of $\FD_{j,k}(200)-\FO_{j,k}(200)$ (midpoint between the last positive and first negative $j$, divided by $200$).}\label{tab:cross}
\end{table}

The existence of \emph{some} transition inside $(a_k,b_k)$ should follow rigorously from a large-deviation principle for $X^D_{n,k}$ and $X^O_{n,k}$ (cf.\ the quasi-power and large-deviation machinery of \cite{HwangYeh} and \cite{PemantleWilsonMelczer}) together with continuity of the rate functions, since the rate function of each statistic vanishes at its own mean and is positive at the other's. The genuinely open part is \emph{uniqueness} of the crossover. These observations do not prove the conjecture.

A related open problem is to characterize those periodic sets of marked part sizes for which the fixed-$j$ profile-word comparison yields an eventual inequality. It may be necessary to compare leading constants when the dominant roots coincide, rather than root locations alone.

\section*{Declaration of generative AI use and data availability}
The author developed the profile-word methodology and the principal fixed-$j$ argument. Generative AI tools were used in the preparation of this manuscript: ChatGPT for \LaTeX\ editorial revisions and to check the algebraic coefficients in \cref{prop:D,prop:O}; Claude to critique an initial structural draft and, later, to review the manuscript and independently reproduce its computations (generating-function expansions checked against brute-force enumeration of all partitions of perimeter $n\le13$; the numerical constants, thresholds and tables in \cref{sec:threshold,sec:further}). The author reviewed all content and takes full responsibility for it.

\end{document}